\documentclass[a4paper,11pt]{amsart}
\usepackage[T1]{fontenc}

\usepackage{amssymb}
\usepackage{textcomp}
\usepackage{amsthm}
\usepackage{color}
\usepackage[all]{xy}
\usepackage{mathtools}

\numberwithin{equation}{section}

\renewcommand{\d}{\mathrm{d}}

\def\k{\Bbbk}

\newcommand{\Aff}{\mathrm{Aff}}
\newcommand{\Inn}{\mathrm{Inn}}
\newcommand{\Hom}{\mathrm{Hom}}

\newcommand{\Fun}{\mathrm{Fun}}
\newcommand{\ab}{\mathrm{ab}}
\newcommand{\trid}{\triangleright}
\newcommand{\id}{\operatorname{id}}

\newcommand{\Ima}{\operatorname{Im}}
\newcommand{\coker}{\operatorname{coker}}
\newcommand{\ord}{\operatorname{ord}}

\newcommand{\Aut}{\operatorname{Aut}}
\newcommand{\End}{\operatorname{End}}

\newcommand{\SL}{\mathbf{SL}}
\newcommand{\Alt}{\mathbb{A}}
\newcommand{\Sym}{\mathbb{S}}

\newcommand{\N}{\mathbb{N}}
\newcommand{\Z}{\mathbb{Z}}
\newcommand{\C}{\mathbb{C}}

\newcommand{\F}{\mathbb{F}}

\newcommand{\eps}{\epsilon}

\renewcommand{\inf}{\operatorname{inf}}
\newcommand{\res}{\operatorname{res}}

\def\pf{\begin{proof}}
\def\epf{\end{proof}}

\def\lg{\langle}
\def\rg{\rangle}
\def\ot{\otimes}

\theoremstyle{plain}
\newtheorem{thm}{Theorem}[section]
\newtheorem{lem}[thm]{Lemma}
\newtheorem{pro}[thm]{Proposition}
\newtheorem{cor}[thm]{Corollary}

\theoremstyle{remark}
\newtheorem{exa}[thm]{Example}
\newtheorem{rem}[thm]{Remark}
\newtheorem*{acknowledgement*}{Acknowledgement}

\begin{document}

\title[The 2nd rack cohomology group]{An explicit description of the second cohomology group of a quandle}
\author[A. Garc\'ia Iglesias, L. Vendramin]{Agust\'in Garc\'ia Iglesias, Leandro Vendramin}

\begin{abstract}
	We use the inflation--restriction sequence and a result of Etingof and
	Gra\~na on the rack cohomology to give a explicit description of
	$2$-cocycles of finite indecomposable quandles with values in an abelian
	group.  Several applications are given.
\end{abstract}

\thanks{This work is partially supported by CONICET, FONCyT PICT-2013-1414 and
PICT-2014-1376, Secyt (UNC), ICTP and MATH-AmSud.}

\address{
Departamento de Matem\'atica -- FCEN,
Universidad de Buenos Aires, Pabell\'on I, Ciudad Universitaria (1428)
Buenos Aires, Rep\'ublica Argentina}
\email{lvendramin@dm.uba.ar}

\address{
FaMAF--CIEM (CONICET) -- Universidad Nacional de C\'ordoba, Medina Allende s/n, 
Ciudad Universitaria (5000) C\'ordoba, Rep\'ublica Argentina.}
\email{aigarcia@famaf.unc.edu.ar}

\maketitle
\setcounter{tocdepth}{1}

\section{Introduction and main results}

\subsection{}

Quandles are non-associative algebraic structures introduced independently by
Joyce~\cite{MR638121} and Matveev~\cite{MR672410} in connection with knot
theory. They produce powerful invariants similar to those obtained by 
coloring~\cite{MR3253967}, \cite{MR3054333}. Quandles turned out to be useful in different branches of algebra, topology and
geometry since they have connections to several different topics such as
permutation groups~\cite{MR3399387}, quasigroups~\cite{MR3353113},  symmetric
spaces~\cite{MR3127819}, Hopf algebras~\cite{MR1994219}, etc. 

Quandles have a very interesting cohomology theory that first appeared
in~\cite{MR1990571} and independently in~\cite{MR1800709}.
This theory is somewhat based on the rack cohomology introduced
in~\cite{MR2063665} by Fenn, Rourke and Sanderson. As in the case of groups,
2nd quandle cohomology groups can be used to produce new
quandles by means of extensions. 

The explicit computation of quandle cohomology groups is an
important problem relevant to different areas of current research.
The 2nd quandle cohomology group is particularly important
since it has many applications going from knot theory to Hopf algebras.

In~\cite{MR1990571}, Carter, Jelsovsky, Kamada, Langford and Saito
 used
quandle cohomology classes to produce powerful invariants of
classical links and their higher dimensional analogs. The
invariants based on quandle $2$-cocycles improve the effectiveness of the
quandle-coloring invariants since, for example, 
they distinguish knots from their mirror images. These invariants require an
explicit description of $2$-cocycles. 

In the Hopf algebra context, quandles and their cohomology parametrize
Yetter-Drinfeld modules. In turn these modules are crucial ingredients in the
classification problem of finite-dimensional Hopf algebras with non-abelian
coradical. Indeed, an important step of the Lifting Method proposed by
Andruskiewitsch and Schneider to solve this classification
problem is the explicit computation of the 2nd cohomology of finite quandles,
see~\cite{MR2799090}. 

\subsection{}

In this work we give an explicit description of the second cohomology
group of a finite indecomposable quandle. Our presentation is made by means of
the characters of a certain finite group. This reduces the problem of computing
$2$-cocycles of a quandle to an easy manipulation involving cosets in a finite
group. Our method is based on a result of Etingof and Gra\~na~\cite{MR1948837}
which relates the 2nd cohomology of a quandle and the first cohomology of an
infinite group. 

\subsection{}
\label{subsection:basics}
We now review the basics of our construction. 
Let $X$ be a finite quandle. Recall that the \emph{enveloping group} of $X$ is the group 
\begin{equation}
	\label{eq:G_X}
	G_X=\lg x\in X:xy=(x\trid y)x\rg.
\end{equation}
Assume that $X$ is indecomposable and fix $x_0\in X$. Under the identification
$\langle x_0\rangle\simeq\Z$ we show in Lemma~\ref{lem:semidirecto} that
$G_X\simeq N_X\rtimes\Z$, where $N_X$ is the commutator group $[G_X,G_X]$ of
$G_X$.  The group $G_X$ acts transitively on $X$ in a natural way, hence so
does $N_X$, see Corollary \ref{cor:semidirecto}.  We denote by $N_0$ the
stabilizer of $N_X$ on $x_0$: this is a finite group \emph{cf.} Lemma \ref{lem:isoclinic}.

Fix an abelian group $A$ and let $M=\Fun(X,A)$ be the right $G_X$-module of
functions $X\to A$, i.e. $(f\cdot x)(y)=f(x\triangleright y)$ for $x,y\in X$
and $f\in M$. 

We prove that there is a commutative diagram with exact columns 
\[
	\xymatrix{
	& 0\ar[d] & 0\ar[d]\\
	& H^1(\Z,M^N)\ar[r]^{\qquad\sim}\ar[d]_{\inf} & A\ar[d]^{\iota}\\
	H^2(X,A)\ar@{-}[r]^{\sim} & H^1(G_X,M)
	\ar[d]_{\res} & A\times H^1(N_X,M)\ar[d]^{\pi}\\
	& H^1(N_X,M)^{\Z}\ar@{=}[r]\ar[d] & H^1(N_X,M)\ar[r]^{\sim}\ar[d] & \Hom(N_0,A)\\
	& 0 & 0
	}
\]
where the isomorphism 
\begin{equation}
	\label{eqn:EG}
	H^2(X,A)\simeq H^1(G_X,M),\quad
	q\longmapsto f_q,
\end{equation}
is~\cite[Corollary 5.4]{MR1948837}, see also \eqref{eqn:fq}; $\inf$ and $\res$
denote the inflation-restriction maps and $\iota$ and $\pi$ denote the
canonical inclusion and projection. 

See Lemmas \ref{lem:1}, \ref{lem:step5} and Proposition \ref{pro:2} for a proof
of the isomorphisms and the equality in the rows of the diagram.  The exactness
of the first column is a well-known fact \emph{cf.} Lemma \ref{lem:Brown}.  We
show that it splits in Lemma \ref{lem:split}. 

By diagram chasing, we derive an isomorphism 
\begin{equation*}
	H^2(X,A)\simeq A\times\Hom(N_0,A).
\end{equation*}
From this isomorphism we obtain an explicit description of rack and quandle
$2$-cocycles with values in any abelian group $A$, see Theorem \ref{thm:main}.

We denote by $f\mapsto f_0$ the map $H^1(G_X,M)\to \Hom(N_0,A)$ deduced from
the diagram above.  

Our first main result reads as follows, see \S
\ref{sec:proof} for a proof.

\begin{thm}
	\label{thm:main}
	Let $X$ be a finite indecomposable quandle, $x_0\in X$ and $A$ an abelian
	group with trivial $G_X$-action. Then
	\begin{equation}
	\label{eqn:map}
		H^2(X,A)\simeq A\times\Hom(N_0,A),
		\quad
		q\mapsto (q_{x_0,x_0},(f_q)_0).
	\end{equation}
\end{thm}
In particular this shows that the non-constant 2-cocycles on $X$ are controlled by a finite group.

\subsection{}
Our second main result is a precise recipe to reconstruct a cocycle $q\in H^2(X,A)$ from a datum 
$(a,g)\in A\times\Hom(N_0,A)$. That is, we give a converse to the map in~\eqref{eqn:map} to
build all explicit 2-cocycles of a given quandle. To do this, we need to introduce some extra notation.

First, we fix a \emph{good coset
decomposition} 
\[
N_X=\bigsqcup_{i=0}^k \sigma_iN_0, \] into $N_0$-cosets, i.e. the 
representatives $\sigma_0,\dots,\sigma_k$ are chosen so that:
\begin{enumerate}
	\item $\sigma_0=1$;
	\item for each $i\in\{0,\dots,k\}$ there is $j\in\{0,\dots,k\}$ such that $x_0\trid
		\sigma_i=\sigma_j$;
	\item for each $x\in X$ there is $j\in\{0,\dots,k\}$ such that $\sigma_j\trid x_0=x$. 
\end{enumerate}

The existence of such a decomposition is given in Proposition~\ref{pro:cosets},
together with a recursive method for constructing it.  

We define 
$\sigma:N\to \{\sigma_0,\dots,\sigma_k\}$, $\sigma(n)=\sigma_i$ if $n\in \sigma_iN_0$. 
We set, 
{\it cf.} \eqref{eqn:c(n)},  
\begin{equation*}
	c(n)=\sigma(n)^{-1}n\in N_0.
\end{equation*}
Given $y\in X$ and $j\in\{0,\dots,k\}$ such that $\sigma_j\trid x_0=y$ we write
\begin{equation*}
\sigma_y\coloneqq\sigma_j.
\end{equation*}

Our second main result is the following, see \S \ref{sec:reconstruction} for the proof.
\begin{thm}
	\label{thm:reconstruction}
	Let $X$ be a finite indecomposable quandle, $x_0\in X$ and $A$ an abelian
	group with trivial $G_X$-action. Let 
	$N_X=\bigsqcup_{i=0}^k \sigma_iN_0$ be 
	a good decomposition of $N_X$ into $N_0$-cosets. For each $a\in A$ and
	$g\in \Hom(N_0,A)$, the map $q\colon X\times X\to A$ given by 
	\begin{align}\label{eqn:cociclo2}
		q_{x,y}=a+g(c(x\sigma_yx_0^{-1}))
	\end{align}
	is a $2$-cocycle of
	$X$ with values in $A$. 
\end{thm}

Combining Theorems~\ref{thm:main},~\ref{thm:reconstruction} and the isomorphism
\eqref{eqn:EG}, namely 
\begin{align*}
 q_{x,y}=f_q(x)(y),\quad
 q\in H^2(X,A),
\end{align*}
we immediately obtain the following corollary.

\begin{cor}
	\label{cor:reconstruction}
	Let $X$ be a finite indecomposable quandle, $x_0\in X$ and $A$ an abelian
	group with trivial $G_X$-action. Let 
	$N_X=\bigsqcup_{i=0}^k \sigma_iN_0$ be 
	a good decomposition of $N_X$ into $N_0$-cosets and let $q\in H^2(X,A)$. Then 
	there exists 
	$a\in A$ and
	$g\in \Hom(N_0,A)$ such that~\eqref{eqn:cociclo2} holds for all $x,y\in X$.
\end{cor}

Corollary~\ref{cor:reconstruction} has many applications and can be used for
explicit calculations of rack cohomology groups of quandles. In particular, if
the commutator subgroup $N_X$ acts regularly on $X$, then $N_0=1$ and hence we
obtain the following corollary.

\begin{cor}
    Let $X$ be a finite indecomposable quandle. If the action of $N_X$ on $X$
    is regular, then $H^2(X,\C^\times)\simeq\C^\times$. 
\end{cor}

\subsection{}
The paper is organized as follows. Preliminaries on racks and quandles,
cohomology theory of groups, and cohomology theories of racks and quandles
appear in Section~\ref{sec:preliminaries}.  Our first main result, 
Theorem~\ref{thm:main}, is proved in Section~\ref{sec:main}. 
Theorem~\ref{thm:reconstruction} is proved in Section~\ref{sec:reconstruction}. 
Applications of our theory 
are given in Section~\ref{sec:applications}.
These applications include the calculations of the 2nd rack cohomology
group of: (a) the quandle associated with the conjugacy class of transpositions, see Theorem~\ref{thm:FK};
(b) affine racks of size $p$ and $p^2$, where $p$ is a prime number, see
Propositions~\ref{pro:affine},~\ref{pro:p},~\ref{pro:homology-affine}
and~\ref{pro:explicit-affine};
and (c) another
proof of Eisermann's formula for computing the 2nd quandle homology group of a
quandle, see Theorem~\ref{thm:eisermann}.

\section{Preliminaries}
\label{sec:preliminaries}

\subsection{Notation}

For a set $X$ we denote by $\Sym_X$ the group of permutations $X\to X$. If $X$
is finite of cardinal $|X|\in\mathbb{N}$, then we identify $\Sym_{|X|}=\Sym_X$.
For any group $G$ we denote by $[G,G]$ its commutator subgroup and $G_{\ab}$
its abelianization, i.e. $G_{\ab}=G/[G,G]$. In addition, $Z(G)$ is the center
of $G$ and $G_G(g)=\{h\in G:hg=gh\}$ for $g\in G$. We denote by $\Aut(G)$ the
group of automorphisms $G\to G$; if $\gamma\in\Aut(G)$, then $\ord(\gamma)$
is the order of $\gamma$. 

Let $M$ be an abelian group equipped with a $G$-action. We denote by
$H^n(G,M)$, $n\geq 0$, the $n$th cohomology group of $G$ with coefficients on
$M$. We denote by $Z^n(G,M)$, resp. $B^n(G,M)$, the groups of cocycles, resp.
cobordisms, of $G$ with values on $M$. We refer the reader to \cite{MR1324339}
for unexplained notation and terminology.

\subsection{Racks}\label{sec:racks}

A \emph{rack} is a non-empty set $X$ together with a binary operation
$\trid:X\times X\to X$ such that the maps $\varphi_x=x\trid -:X\to X$,
$y\mapsto x\trid y$, are bijective for each $x\in X$, and $x\trid (y\trid
z)=(x\trid y)\trid (x\trid z)$ for all $x,y,z\in X$. A {\it quandle} is a rack
that further satisfies $x\trid x=x$ for all $x\in X$. A prototypical example of
a rack is a group $G$ with $\trid$ given by conjugation. A rack is
\emph{indecomposable} if the inner group $\Inn(X)=\langle\varphi_x:x\in
X\rangle\leq\Sym_X$ acts transitively on $X$. 

The enveloping group $G_X$ \emph{cf.}~\eqref{eq:G_X} also acts on $X$, and this
action is readily seen to be transitive when  $X$ is indecomposable. The group
$G_X$ is infinite. There is a finite analogue of this group, which is
constructed as follows: For each $x$, let $n_x=\ord\varphi_x$. Then the
subgroup $Z_X=\lg x^{n_x}, x\in X\rg\leq G_X$ is normal and the 
quotient $F_X=G_X/Z_X$ is finite,
see~\cite[\S2]{MR2803792}. We write $N_X=[G_X,G_X]$ to denote the commutator
subgroup of $G_X$. 

\begin{lem}{\cite[Lemma 1.10]{MR3276225}}
	\label{lem:isoclinic}
	Let $X$ be an indecomposable quandle. Then $N_X\simeq[F_X,F_X]$. In particular,
	$N_X$ is finite. 	
\end{lem}

The last claim of Lemma~\ref{lem:isoclinic} also follows from the following
result and a theorem of Schur, see for
example~\cite[Theorem 5.32]{MR1307623}. 

\begin{lem}
	\label{lem:Schur}
	Let $X$ be a finite indecomposable quandle. Then all conjugacy classes of
    $G_X$ are finite.
\end{lem}

\begin{proof}
	Since $G_X$ acts transitively on $X$ and the center $Z(G_X)$ is the kernel
	of this action, it follows that the index $[G_X:Z(G_X)]$ is finite. This
	implies that all conjugacy classes of $G_X$ are finite as 
	\[	
	[G_X:C_{G_X}(g)]\leq [G_X:Z(G_X)],
	\]	
    where $C_{G_X}(g)$ denotes the centralizer of $g$ in $G_X$.
\end{proof}

We consider the unique surjective group homomorphism 
\begin{align}\label{eqn:d}
 \d:G_X\to \Z
\end{align}
satisfying $\d(x)=1$ for all $x\in X$. In particular, this homomorphism shows
that $G_X$ is infinite and induces a notion of degree on $G_X$. 

\begin{lem}
	\label{lem:semidirecto}
	Let $X$ be an indecomposable finite quandle and $x_0\in X$. Then the following
	hold:
	\begin{enumerate}
		\item $G_X=\ker\d\rtimes\langle x_0\rangle$.
		\item $\ker\d=N_X$ if $X$ is indecomposable. 
	\end{enumerate}

	\begin{proof}
		Since $\ker\d$ is a normal subgroup of $G_X$, $\ker\d\langle
		x_0\rangle$ is a subgroup of $G_X$. It is clear that $\ker\d\cap\langle
		x_0\rangle=1$ comparing degrees. Finally $G_X=\ker\d\langle x_0\rangle$
		since $x=(xx_0^{-1})x_0\in\ker\d\langle x_0\rangle$ for all $x\in X$.

		It is clear that $N_X\subseteq\ker \d$. Next we prove the equality when
		$X$ is indecomposable.  Let $\ell:G_X\to \Z$ be defined as $\ell(g)=n$,
		if $g=x_{i_1}^{\eps_1}\dots x_{i_n}^{\eps_n}$, $\eps_i\in\{\pm1\}$,
		$i\in\{1,\dots,n\}$, is a a reduced expression of $g$ in terms of the
		generators of $G_X$. We show that $\ker \d\subseteq N_X$ by induction
		on  $\ell(g)$, $g\in\ker \d$.  If $\ell(g)=2$, then $g=x_i^{\pm
		1}x_j^{\mp1}$. So we may assume that $g=x_ix_j^{-1}$ (if not, take
		inverses).  Now, as $X$ is indecomposable, there is $h\in G_X$ such
		that $h\cdot x_j=x_i$.  Hence $g=hx_jh^{-1}x_j^{-1}\in N_X$. Now, if
		$\ell(g)>2$, then there is a reduced expression of $g$ (or $g^{-1}$) in
		which $g=g_1x_ix_j^{-1}g_2$, $x_i,x_j\in X$ and $g_1,g_2\in G_X$.  Now,
		on the one hand, $0=\d(g)=\d(g_1)+\d(g_2)$ and thus $g_1g_2\in N_X$ as
		$\ell(g_1g_2)<\ell(g)$. On the other,
		$g=(g_1x_ix_j^{-1}g_1^{-1})(g_1g_2)$ and therefore $g\in N_X$.  
	\end{proof}
\end{lem}

\begin{cor}
	\label{cor:semidirecto}
	The restriction of the action of $G_X$ on $X$ to $N_X$ is transitive.

	\begin{proof}
		Let $x,y\in X$ and let $g\in G_X$ such that $g\cdot x=y$ and let
		$\ell=\d(g)$.  Then $g'=y^{-\ell}g\in N_X$ by Lemma \ref{lem:semidirecto} and
		$g'\cdot x=y$.
	\end{proof}
\end{cor}

\subsection{Rack cohomology}

A cohomology theory for racks was introduced in~\cite{MR1364012} and
independently in~\cite{MR1800709}. A cohomology theory for quandles was
developed in~\cite{MR1990571}. These theories were further developed and
generalized for example in \cite{MR1994219} and~\cite{MR3081627}.

We briefly recall these cohomology theories next. Let $X$ be a rack and let $M$ be a right $G_X$-module. 
Set $C^n=C^n(X,M)=\Fun(X^n,M)$, $n\geq 0$, the set of functions from $X^n$ to $M$. 
Consider the differential $d:C^n\to C^{n+1}$
\begin{multline*}
df(x_1,\dots, x_{n+1})=\sum_{i=1}^{n}(-1)^{i-1}\Big(f(x_1,\dots, x_{i-1},x_{i+1},\dots, x_{n+1})\\
-f(x_1,\dots, x_{i-1},x_i\trid x_{i+1},\dots, x_i\trid x_{n+1})\cdot x_i\Big).
\end{multline*}
The \emph{rack cohomology} $H^\bullet(X,M)$ of $X$ with coefficients in $M$ is
the cohomology of the complex $(C^\bullet,d)$ \cite[Definition 2.3]{MR1948837}.
The groups of cocycles resp. cobordisms, are denoted by $Z^n(X,M)$, resp.
$B^n(G,M)$.  When $A$ is an abelian group and no reference to a $G_X$-action on
$A$ is specified, $H^\bullet(X,A)$ stands for the cohomology of $X$ with values in the
trivial module $M=A$. If $q$ is a class in $H^2(X,A)$, we set $q_{x,y}\coloneqq
q(x,y)$. Hence $q\in H^2(X,A)$ if and only if 
\begin{align}\label{eqn:cocycle-condition}
q_{x\trid y,x\trid z}q_{x,z}=q_{x,y\trid z}q_{y,z}, \qquad \forall\,x,y,z\in X
\end{align}
and two classes $q,q'\in H^2(X,A)$ are equivalent if and only if there exists
$\gamma:X\to A$ such that $q'_{x,y}=q_{x,y}\gamma(x\trid y)\gamma(y)^{-1}$ for
all $x,y\in X$.

The \emph{rack homology} $H_\bullet(X,A)$ with values in an abelian group $A$ is
defined analogously, by considering the free abelian group $F_n(X)$ on $X^n$,
$n\geq 0$, and setting $C_n(X,A)\coloneqq F_n(X)\ot A$. If $X$ is a quandle,
then the subgroup $F_n^D(X)\leq F_n(X)$ generated by $n$-tuples $(x_1,\dots
x_n)$ with $x_{i}=x_{i+1}$ for some $i$, defines a subcomplex
$C_\bullet^D=C_\bullet^D(X,A)$ of $C_\bullet$. The \emph{quandle homology}
$H_\bullet^Q(X,A)$ of $X$ is the homology of the quotient complex
$C_\bullet^Q=(C_n/C_n^D)_{n\geq 0}$.

%

\medbreak
In this work we give a description of the group $H^2(X,A)$ of 2-cocycles on $X$ 
with values in an abelian group $A$, which allows us to compute cocycles explicitly. 
We recall next some identifications between the (co)homology theories described above that will be useful for our goal.

\begin{lem}{\cite[Proposition 3.4]{MR1812049}}  
	\label{lem:UCT}
	$H^2(X,A)\simeq\Hom(H_2(X,\Z),A)$, via 
\[
H^2(X,A)\ni q\mapsto
\left([x,y]\mapsto q_{x,y}\right)\in \Hom(H_2(X,\Z),A).
\]
\end{lem}

The following is a particular case of \cite[Theorem 7]{MR1952425}.
\begin{lem}
	\label{lem:H2}
	Assume $X$ is an indecomposable quandle. Then 
	\[
	H_2(X,\Z)\simeq H_2^Q(X,\Z)\times \Z.
	\]
	Explicitly, if $(x,y)\in X^2$, then the isomorphism is induced by the map 
	\begin{align*}
	(x,y)\mapsto \begin{cases}
	(x,y)\times 0, & \text{ if }x\neq y\\
	0\times 1, & \text{ if }x= y.
	\end{cases}
	\end{align*}
\end{lem}

%
%
%
%
%
%
%
%

Etingof and Gra\~na found a deep relation between group cohomology and rack
cohomology. 

\begin{thm}{\cite[Corollary 5.4]{MR1948837}}
	\label{thm:EG}
	Let $X$ be a finite indecomposable rack and $A$ an abelian group with a
	trivial $G_X$-action. Then 
	\begin{equation*}
		H^1(G_X,\Fun(X,A))\simeq H^2(X,A).
	\end{equation*}
	This equivalence is given as follows:
	\begin{enumerate}
		\item If $f\in H^1(G,\Fun(X,A))$ then a
		2-cocycle $q^f\in H^2(X,A)$ arises as 
		\[
		q^f_{x,y}=f(x)(y), \qquad x,y\in X.
		\]
			\item 	Conversely, $q\in H^2(X,A)$ determines $f_q\in  H^1(G,\Fun(X,A))$ by extending
		$q$ recursively via 
		\begin{align}\label{eqn:fq}
		f_q(xy)(z)=q_{x,y\trid z}+q_{y,z}, \quad x,y,z\in X. 
		\end{align}
	\end{enumerate} 
\end{thm}

\begin{rem}
Let $G$ be a (non-abelian) group and fix $Z^2(X,G)\subset \Fun(X^2,G)$ as the subset of all $q:X^2\to G$ satisfying \eqref{eqn:cocycle-condition}.
We say that $q$ is equivalent to $q'$, and we write $q\sim q'$, in $Z^2(X,G)$ if and only if there is $\gamma\in \Fun(X,G)$ such that $q'_{x,y}=\gamma(x\trid y)q_{x,y}\gamma(y)^{-1}$. 
If $H^2(X,G)\coloneqq Z^2(X,G)/\sim$, then Theorem \ref{thm:EG} holds, see \cite[Remark 5.6]{MR1948837}.
\end{rem}

 \subsection{Group cohomology}

\label{subsection:cohomology}

Let $G$ be a group, $N\vartriangleleft G$ a normal subgroup and $M$ a right
$G$-module. Recall \emph{cf.} \cite[3.8]{MR1324339} that there is a right $G/N$-action on $H^1(N,M)$, induced by 
\begin{align}\label{eqn:action-invariants}
(f\cdot g)(n)= f(gng^{-1})\cdot g, \quad g\in G, n\in N, f\in H^1(N,M).
\end{align}
Indeed, let $f\in Z^1(N,M)$. If $g\in N$, then \[(f\cdot g)(n)=f(gn)-f(g)=f(g)\cdot n+f(n)-f(g)\] by the cocycle condition. Hence
\begin{align*}
(f\cdot g)(n)-f(n)&=f(g)\cdot n-f(g)
\end{align*}
and thus $f\cdot g=f\in H^1(N,M)$. The {\it inflation-restriction} sequence is 
\begin{multline}\label{eqn:infl-restr} 
	0\to
	H^1(G/N,M^N)\overset{\iota}{\to}H^1(G,M)\overset{r}{\to}H^1(N,M)^{G/N}\\
	\to H^2(G/N,M^N)\to H^2(G,M) 
\end{multline} 
where the {\it inflation map} $\iota(h)$, $h\in H^1(G/N,M^N)$, is the
composition $$G\twoheadrightarrow G/N\overset{h}{\to}M^N\hookrightarrow M$$ and
the {\it restriction map} $r(g)$, $g\in H^1(G,M)$, is the composition
$$N\hookrightarrow G\overset{g}{\to}M.$$

In the case where $G/N\simeq \Z$ one obtains the following result, see \emph{loc.cit.} 

\begin{lem}
	\label{lem:Brown}
	Assume that $G/N\simeq\Z$. Then 
	\begin{enumerate}
		\item\label{item:H1} $H^2(G/N,M^N)=0$.
		\item\label{item:H2} $H^1(G/N,M^N)=M^N/\langle m\cdot g - m\rangle$,
			(class of) $f\mapsto \text{ (class of) }f(1)$.
	\end{enumerate}
	In particular, the exact sequence \eqref{eqn:infl-restr} reduces to 
	\begin{align}\label{eqn:sec}
		0\to H^1(G/N,M^N)\overset{\inf}{\to}H^1(G,M)\overset{\res}{\to}H^1(N,M)^{G/N}\to 0.
	\end{align}
\end{lem}

\begin{lem}
	\label{lem:split}
	Assume that $N$ is finite and $G/N\simeq\Z$. Then~\eqref{eqn:sec} splits. 
	A retraction for  $\inf:H^1(G/N,M^N)\to H^1(G,M)$ is given by
	\begin{align*}
		&j:H^1(G,M)\to H^1(G/N,M^N), && 
		j(f)(\ell)=\frac{1}{|N|}\sum_{n\in N}\left(f(nx_0^\ell)-f(n)\right).
	\end{align*}
\end{lem}


\begin{proof}
	We need to check that $j$ is well-defined, that is:
	\begin{enumerate}
		\item If $f\in Z^1(G,M)$, then $j(f)(G/N)\subseteq M^N$.
		\item If $f\in Z^1(G,M)$, then $j(f)\in Z^1(G/N,M^N)$.
		\item If $f\in B^1(G,M)$, then $j(f)\in B^1(G/N,M^N)$.
	\end{enumerate}
	Let $f\in Z^1(G,M)$ and set $\varphi\coloneqq j(f)$.
	For (1), using the cocycle condition, 
	\begin{align*}
		\varphi(\ell)\cdot n&=\frac{1}{|N|}\sum_{m\in N}\left(f(mx_0^\ell)\cdot n
		-f(m)\cdot n\right)\\
		&=\frac{1}{|N|}\sum_{m\in N}\left(f(mx_0^\ell n)-f(n)-f(mn)+f(n)\right)\\
		&=\frac{1}{|N|}\sum_{m\in N}\left(f(mx_0^\ell 
		nx_0^{-\ell}x_0^\ell)-f(mn)\right). 
	\end{align*}
	By reordering the sum, $\varphi(\ell)\cdot n=\varphi(\ell)$ for all 
	$n\in 	N$, $\ell\in\Z$. Hence (1) holds. 

	In (2), we get
	\begin{align*}
		\varphi(\ell+r)&=\frac{1}{|N|}\sum_{n\in N}\left(f(nx_0^{\ell+r})-f(n)\right)\\
		&=\frac{1}{|N|}\sum_{n\in N}\left(f(nx_0^\ell)\cdot x_0^r+f(x_0^r)-f(n)\right)\\
		&=\frac{1}{|N|}\sum_{n\in N}\left(f(nx_0^\ell)\cdot x_0^r-f(n)\cdot x_0^r+f(n)\cdot x_0^r+f(x_0^r)-f(n)\right)\\
		&=\varphi(\ell)\cdot r+\frac{1}{|N|}\sum_{n\in N}\left(f(nx_0^r)-f(x_0^r)+f(x_0^r)-f(n)\right)\\
		&=\varphi(\ell)\cdot r+\varphi(r).
	\end{align*}
	Thus (2) holds.
	If $f\in B^1(G,M)$, then there exists $\psi\in M$ such that 
	$f(g)=\psi\cdot 	g-\psi$. Hence, 
	\begin{align*}
		j(f)(\ell)&=\frac{1}{|N|}\sum_{n\in N}\left(\psi\cdot nx_0^\ell-\psi-\psi\cdot 
		n+\psi\right)\\
		&=\frac{1}{|N|}\sum_{n\in N}\left(\psi\cdot nx_0^\ell-\psi\cdot n\right)
	\end{align*}
	and thus $j(f)(\ell)=\gamma\cdot\ell-\gamma$ for 
	\begin{align*}
		\gamma=\frac{1}{|N|}\sum_{n\in N}\psi\cdot n\in M^N.
	\end{align*}
	This shows (3). 	Finally we prove that $j\circ\inf=\id$. 
For this, recall that 
	if $\varphi\in 	H^1(G/N,M^N)$ and $g\in G$, then 
	$\inf(\varphi)(g)=\varphi(\bar g)$,  where $\bar g$ is the class of $g$ in $G/N\simeq\Z$. Then 
	\begin{align*}
	(j\circ\inf)(\varphi)(\ell)&=
	\frac{1}{|N|}\sum_{n\in N}\left(
	\inf(\varphi)(nx_0^\ell)-\inf(\varphi)(n)\right)=\frac{1}{|N|}\sum_{n\in N}
	\varphi(\ell)=\varphi(\ell)
	\end{align*}
	for all $\ell\in\Z$.
	This completes the proof.
\end{proof}

\section{Proof of Theorem~\ref{thm:main}}
\label{sec:main}

Assume that $X$ is a finite indecomposable rack.  We write $G=G_X$,
$N=[G_X,G_X]$. Let $A$ be an abelian group with trivial $G$-action and set
$M=\Fun(X,A)$. Fix $x_0\in X$ and $G\simeq N\rtimes\Z$ as in Lemma~\ref{lem:semidirecto}.
It follows
from Lemma~\ref{lem:split} that
\[
	0\to H^1(G/N,M^N)\xrightarrow{\inf}H^1(G,M)\xrightarrow{\res}H^1(N,M)^{G/N}\to 0
\]
splits.  We first identify the first term of this sequence.

\begin{lem}
	\label{lem:1}
	$H^1(\Z,M^N)\simeq A$, via $f\mapsto f(1)(x_0)$.

	\begin{proof}
		Recall from Lemma~\ref{lem:Brown}(2) that
		$H^1(\Z,M^N)\simeq M^N/F$, where $F$ is the submodule generated by
		$\{\varphi\cdot x_0^p-\varphi: p\in\Z, \varphi\in M^N\}$. Since 
		$X=N\trid \{x_0\}$ by Corollary~\ref{cor:semidirecto} and  
		$n\trid x_0=x\in X$ for some $n\in N$, 
		$$ 		\varphi(x)=\varphi(n\trid x_0)=(\varphi\cdot n)(x_0) 
		$$
		for all $\varphi\in M$.  Hence, if $\varphi\in M^N$, then
		$\varphi(x)=\varphi(x_0)$, $x\in X$.  Consequently, $F=\{0\}$ and
		$H^1(\Z,M^N)\simeq M^N$. But $M^N\simeq A$ as any $\varphi\in M^N$ is
		determined by its value $\varphi(x_0)\in A$. Hence the lemma follows.
	\end{proof}
\end{lem}

As for the third term, we will show in Proposition~\ref{pro:2} that 
\begin{align}\label{eqn:third-term}
H^1(N,M)^\Z\simeq  \Hom(N_0,A). 
\end{align}
To do so, we first need several lemmas. 

\begin{lem}
	\label{lem:step1}
	The map
	\begin{equation}
		\label{eqn:step1}
		Z^1(N,M)\to \Hom(N_0,A), 
		\quad f\mapsto f_0,
	\end{equation}
	where $f_0(n_0)=f(n_0)(x_0)$ for $n_0\in N_0$, is well-defined
	and factors to a map $H^1(N,M)\to \Hom(N_0,A)$.

	\begin{proof}
		We first prove that $f_0$ is indeed a group homomorphism:
		\begin{align*}
			f_0(n_0n_0')&=f(n_0n_0')(x_0)=(f(n_0)\cdot n_0')(x_0) + f(n_0')(x_0)\\
			&=f(n_0)(n_0'\trid x_0) + f(n_0')(x_0)\\
			&=f(n_0)(x_0) + f(n_0')(x_0)=f_0(n_0) + f_0(n_0'), \ n_0,n_0'\in N_0.
		\end{align*}

		We now show that the map factors to a map $H^1(N,M)\to \Hom(N_0,A)$. 
		Let $f\in B^1(N,M)$, that is $f(n)=\varphi\cdot n-\varphi$ for some $\varphi\in M$. Then
		\begin{align*}
			f_0(n_0)&=f(n_0)(x_0)=(\varphi\cdot n_0)(x_0)-\varphi(x_0)\\
			&=\varphi (n_0\trid x_0)-\varphi(x_0)=\varphi (x_0)-\varphi(x_0)=0.
		\end{align*}
		This completes the proof.
	\end{proof}
\end{lem}

\begin{lem}
	\label{lem:step3}
	\label{lem:step4}
	The map $H^1(N,M)\to \Hom(N_0,A)$, $f\mapsto f_0$, is an injective group
	homomorphism.

	\begin{proof}
		It is clear that $f\mapsto f_0$ is a group homomorphism. 

		Let $f\in H^1(N,M)$ be such that $f_0=0$. That is, $f(n_0)(x_0)=0$ for
		every $n_0\in N_0$. We claim that there is $\varphi\in M$ such that
		$f(m)=(\varphi\cdot m)-\varphi$ and thus $f=0$ in $H^1(N,M)$. Set 
		\begin{equation*}
			\varphi(x)\coloneqq f(n)(x_0)\qquad \text{ if } x=n\trid x_0.
		\end{equation*}

		Let us check that this is well-defined: if $x=n\trid x_0=n'\trid x_0$, then
		$n^{-1}n'\in N_0$. Since $f(1)=0$, one obtains that $f(n^{-1})=-f(n)\cdot
		n^{-1}$. Then 
		\begin{align*}
			0&=f_0(n^{-1}n')=f(n^{-1}n')(x_0)\\
			&=-f(n)(n^{-1}n'\trid x_0)+f(n')(x_0)
			=-f(n)(x_0)+f(n')(x_0),
		\end{align*}
		and thus $\varphi(x)$ does not depend on $n\in N$ such that $x=n\trid
		x_0$.  Finally for each $m\in N$ and every $x=n\trid x_0\in
		X$ with $n\in N$,
		\begin{align*}
			(\varphi&\cdot m-\varphi)(x)=\varphi(m\trid x)-\varphi(x)=\varphi(m\trid n\trid x_0)-\varphi(n\trid x_0)\\
			&=f(mn)(x_0)-f(n)(x_0)=(f(m)\cdot n)(x_0)+f(n)(x_0)-f(n)(x_0)\\
			&=f(m)(n\trid x_0)=f(m)(x),
		\end{align*}
		and therefore $f=0$.
	\end{proof}
\end{lem}
Recall the definition of the $\Z$-action on $H^1(N,M)$ from \eqref{eqn:action-invariants}.
\begin{lem}
	\label{lem:step5}
	Assume $X$ is a quandle. Then
	$H^1(N,M)=H^1(N,M)^\Z$.

	\begin{proof}
		Let $f\in H^1(N,M)$ and set $g=f-f\cdot x_0$. If $n_0\in N_0$, then 
		\begin{align*}
			g_0(n_0)=f(n_0)(x_0)-f(x_0n_0x_0^{-1})(x_0\trid x_0)=0.
		\end{align*}
		Thus $g_0=0$ and hence $f=f\cdot x_0$ for all $f\in H^1(N,M)$ by
		Lemma~\ref{lem:step4}, since the group homomorphism $g\mapsto g_0$ is
		injective.
	\end{proof}
\end{lem}

In order to show the surjectivity of the map $f\mapsto f_0$ from Lemma \ref{lem:step3}, we need to fix a 
decomposition
of $N$ into $N_0$-cosets 
\[
N=\bigsqcup_{i=0}^k \sigma_iN_0,
\] 
where $\sigma_i\in N$ is a representative, $\sigma_0N_0=N_0$. 
We define 
\begin{align}
\label{eqn:sigma(n)}
\sigma:N\to \{\sigma_0,\dots,\sigma_k\},
\quad
\sigma(n)=\sigma_i\quad\text{if $n\in \sigma_iN_0$}.
\end{align} 
For $n\in N$ we consider $c(n)\in N_0$ defined by 
\begin{align}
\label{eqn:c(n)}
n=\sigma(n)c(n).
\end{align}

\begin{rem}
	\label{rem:c_casi_es_morfismo}
	For all $n\in N$ and $n_0\in N_0$ it follows that $c(nn_0)=c(n)n_0$.
	Indeed, $nn_0=\sigma(n)c(n)n_0=\sigma(nn_0)c(nn_0)$ and thus the claim
	holds since each $m\in N$ decomposes uniquely as $m=\sigma(m)c(m)$. 
\end{rem}

\begin{lem}
	\label{lem:step6}
	The map $H^1(N,M)\to \Hom(N_0,A)$, $f\mapsto f_0$, is surjective. 

	\begin{proof}
		Let $g\colon N_0\to A$ be a group homomorphism; we shall construct an
		$f\in Z^1(N,M)$ such that $f_0=g$.  
		We claim that the map $f:N\to M$, $n\mapsto f(n)$, given by 
		\begin{align}\label{eqn:reconstruction}
			f(n)(x)=g(c(nm))-g(c(m))=g(c(nm)c(m)^{-1}), 
		\end{align} 
		where $m\in N$ is such that 
		$x=m\trid x_0$, 
		is well-defined. 
		Indeed, if $m'\in N$ also satisfies $x=m'\trid x_0$,
		then $m^{-1}m'\in N_0$ and thus $\sigma(m)^{-1}\sigma(m')\in N_0$. That
		is $\sigma(m)=\sigma(m')$ and thus $\sigma(nm)=\sigma(nm')$ for
		every $n\in N$ since $(nm')^{-1}nm\in N_0$. As $g$ is a group homomorphism, 
		\begin{align*}
			g(c(nm)c(m)^{-1})-g(c(nm')c(m')^{-1})&=g(c(nm)c(m)^{-1}c(m')c(nm')^{-1}).
		\end{align*}
		Now, $c(nm)c(m)^{-1}c(m')c(nm')^{-1}$ is, by definition,
		\begin{align*}
			(\sigma(nm)^{-1}nm)(m^{-1}\sigma(m))(\sigma(m')^{-1}m')(m'\,^{-1}n^{-1}\sigma(nm'))=1.
		\end{align*}
		Hence $g(c(nm)c(m)^{-1})-g(c(nm')c(m')^{-1})=g(1)=0$ and thus $f$ does not depend on the choice of $m$.
		
		Now we show that $f\in Z^1(N,M)$. Let $x\in X$, $n,n'\in N$ and $m\in N$ be such that
		$x=m\trid x_0$. On the one hand, we have 
		\begin{align*}
			f(nn')(x)&=g(c(nn'm))-g(c(m)).
		\end{align*}
		On the other,
		\begin{align*}
			(f(n)\cdot n')(x) + f(n')(x)&=f(n) (n'\trid x) + f(n')(x)\\
			&=g(c(nn'm))-g(c(n'm))+g(c(n'm))-g(c(m))\\
			&=f(nn')(x).
		\end{align*}
		Finally we see that $g=f_0$, that is $f_0(n)=g(n)$ for $n\in N_0$. 
		Now, if $n\in N_0$, then 
				$c(n)=c(1\cdot n)=c(1)n$ \emph{cf.} Remark \ref{rem:c_casi_es_morfismo}. Also, as as $x_0=1\trid x_0$,
			\begin{align*}
			f_0(n)&=f(n)(x_0)=g(c(n\cdot 1)c(1)^{-1})=g(c(1\cdot n)) -g (c(1))\\
			&=g(c(1)n) -g (c(1))=g(c(1)) +g(n) -g (c(1))=g(n)
			\end{align*}
			and the lemma follows.
	\end{proof}
\end{lem}

Now we proceed to show \eqref{eqn:third-term}.

\begin{pro}\label{pro:2}
	The map $Z^1(N,M)\to \Hom(N_0,A)$ given by  
	$f\mapsto f_0$, 
	where $f_0(n_0)=f(n_0)(x_0)$ for $n_0\in N_0$, induces
	a group isomorphism 
	\[
	H^1(N,M)^\Z\to \Hom(N_0,A).
	\]

	\begin{proof}
		Lemma~\ref{lem:step5} implies that $H^1(N,M)^{\Z}\simeq H^1(N,M)$ and
		Lemmas~\ref{lem:step4} and~\ref{lem:step6} yield
		$H^1(N,M)\simeq\Hom(N_0,A)$, as desired. 
	\end{proof}
\end{pro}

This allows us to complete the proof of Theorem~\ref{thm:main}.

\subsubsection*{Proof of Theorem~\ref{thm:main}}
\label{sec:proof}
	Using the cocycle condition, we get
	\begin{align*}
		j(f)(\ell)&=\frac{1}{|N|}\sum_{n\in N}\left( f(n)\cdot x_0^{\ell}+f(x_0^\ell)-f(n)\right)\\
		&=f(x_0^\ell)+\frac{1}{|N|}\sum_{n\in N}\left(f(n)\cdot x_0^\ell-f(n)\right).
	\end{align*}
	Hence, as $X$ is a quandle, 
	for each $\ell\in\Z$, 
	\begin{equation}
		j(f)(\ell)(x_0)=f(x_0)(x_0).
		\label{eqn:j(f)(1)}
	\end{equation}

	Since $H^1(G/N,M^N)\simeq A$ by Lemma~\ref{lem:1} and by Proposition~\ref{pro:2}
	there exists an isomorphism $\zeta\colon H^1(N,M)^{\Z}\simeq\Hom(N_0,A)$, we write the
	inflation-restriction sequence~\eqref{eqn:sec} as 
	\begin{align}
		\label{eqn:sec2}
		0\to A\xrightarrow{\inf_0}H^1(G,M)\xrightarrow{\res_0}\Hom(N_0,A)\to 0,
	\end{align}
	where $\res_0(f)=\res(f)_0$ for all $f\in H^1(G,M)$ and  
	$\inf_0$ is the composition $A\simeq M^N\simeq H^1(G/N,M^N)$. 
	We set 
	$f_0=\res_0(f)$
	by abuse of notation, i.e. 
	\begin{align}\label{eqn:f_0}
		f_0(n_0)=f(n_0)(x_0), \quad n_0\in N_0.
	\end{align}

	A retraction for $\inf_0$ is given by the composition 
	$$
	j_0\colon H^1(G,M)\overset{j}{\to} H^1(G/N,M^N)\simeq A,
	$$
	using Lemmas~\ref{lem:split} and \ref{lem:1}, that is 
	\begin{align}
		\label{eqn:f(x_0)}
		j_0(f)=j(f)(1)(x_0)=f(x_0)(x_0),
	\end{align}
	 \emph{cf.}~\eqref{eqn:j(f)(1)}. Hence $H^1(G,M)\simeq A\times\Hom(N_0,A)$ via 
	\begin{align}\label{eqn:correspondence-proof}
		f\mapsto (f(x_0)(x_0),f_0).
	\end{align}
This completes the proof. \qed

\section{Proof of Theorem~\ref{thm:reconstruction}}
\label{sec:reconstruction}

In this section we show the Reconstruction Theorem \ref{thm:reconstruction}.
We fix $x_0\in X$ and write $N_0\leq N_X$ for the
stabilizer of $x_0$ in $N_X$. By Lemma~\ref{lem:isoclinic},
$N_0$ is a finite group. 

The key for the proof of Theorem~\ref{thm:reconstruction} lays in the existence
of a particular class of decompositions \[
N_X=\bigsqcup_{i=0}^k \sigma_iN_0\]
of $N_X$ into $N_0$-cosets, which are \emph{good} in our context. 

\begin{pro}
	\label{pro:cosets}
	Let $X$ be a finite indecomposable quandle. Then there exists a
	decomposition $N_X=\bigsqcup_{i=0}^k \sigma_iN_0$ of $N_X$ into $N_0$-cosets
	such that the following hold:
	\begin{enumerate}
		\item $\sigma_0=1$.
		\item For each $i\in\{0,\dots,k\}$ there is $j\in\{0,\dots,k\}$ such that $x_0\trid
			\sigma_i=\sigma_j$.
		\item For each $x\in X$ there is $j\in\{0,\dots,k\}$ such that $\sigma_j\trid x_0=x$.
	\end{enumerate}
\end{pro}

\pf
Fix a decomposition into cosets $N_X=\bigsqcup_{i=0}^k \sigma_iN_0$. 
Recall from \eqref{eqn:sigma(n)} and \eqref{eqn:c(n)}
the definition of the corresponding assignments 
\[
\sigma:N\to\{\sigma_0,\dots,\sigma_k\} \quad  \text{ and } \quad  c:N\to N_0.
\]

Since
$\sigma_0=1$, (1) holds. Condition (3) also holds trivially: If $x\in X$, there
is $n\in N_X$ is such $n\trid x_0=x$ by Corollary~\ref{cor:semidirecto}.  Now,
there is $j\in\{0,\dots,k\}$ such that $n\in \sigma_jN_0$, that is
$n=\sigma_jn_0$ for some $n_0\in N_0$.  Then $x=n\trid x_0=\sigma_j\trid
(n_0\trid x_0)=\sigma_j\trid x_0$.

For Condition (2), set $S=\{\sigma_1,\dots,\sigma_k\}$. We define
\begin{align*}
	t_j= t_j(S)&\coloneqq\min\{t\geq 1:x_0^t\trid\sigma_j=\sigma_j\}
\end{align*}
for all $j\in\{1,\dots,k\}$. Observe that $1\leq t_j(S)\leq \ord \varphi_{x_0}$, cf.~\S \ref{sec:racks}. For $i\in\{1,\dots,k\}$ and
$t\in\{0,\dots,t_j(S)-1\}$ we define $\tau_{j,t}=x_0^t\trid\sigma_j$ and let 
\[
T=\{\tau_{j,t}:1\leq j\leq k,\;1\leq t< t_j(S)\}.
\]
It is clear that $S\subseteq T$, as $\sigma_j=\tau_{j,0}$ by definition, and that if $S=T$, then we are done.
Notice that this is not a multi-set: we may have $\tau_{j,t}=\tau_{j',t'}$, for different $(j,t), (j',t')$. On the other hand, if $t\neq t'$, then $\tau_{j,t}\neq \tau_{j,t'}$ for  every $j$, since $t<t_j(S)$. In other words, there are $r\leq k$,  $1=i_1<i_2<\dots <i_r$ and $s_j\leq t_{i_j}$, $1\leq j\leq r$ such that
\[
T=\{\tau_{i_j,t}:1\leq j\leq r,\;1\leq t< s_j \}
\]
and $\tau_{i_j,t}\neq \tau_{i_{j'},t'}$ if $j\neq j'$ or $t\neq t'$. We reorder the set $S$ so $i_j=j$, $j=1,\dots,r$. If $S\neq T$, then we proceed inductively:  we order $T$ by:
\[
\tau_{i,s}\prec\tau_{j,t}\Longleftrightarrow  i<j\text{ or }i=j\text{ and }s<t.
\]
Let $\tau=\min\{\tau_{j,t}:\tau_{j,t}\not\in S\}$ and let $\ell$ be such that
$\sigma(\tau)=\sigma_\ell$, i.e.~$\tau_{j,t}=x_0^t\trid \sigma_j\in\sigma_\ell N$ and $\tau_{j,t}\neq \sigma_\ell$. Observe that if $\tau=\tau_{j,t}$, then $\ell\ne
j$.  Set $S_0=S$ and $T_0=T$. We make a new choice of representatives replacing
the original set $S_0$ by 
\[
S_1=\left(S_0\setminus\{\sigma_\ell\}\right)\cup\{\tau\}=\{\sigma_1,\dots,\sigma_{\ell-1},\tau,\sigma_{\ell+1},\dots,\sigma_k \}.
\]
Define $t_j(S_1)$ and $(T_1,\prec)$ accordingly. We claim that $t_j(S_1)\leq t_j(S_0)$
for all $j$. Indeed, equality holds if $j\ne\ell$ and it readily follows that 
\[
t_\ell(S_1)=t_\ell(S_0)-t<t_\ell(S_0). 
\]
In particular, it follows that $|S|=|S_1|\leq |T_1|<|T_0|$. (This also follows as when
constructing $T_1$ we are removing all the $\tau_{\ell,t}$.) If $T_1=S_1$, then
we are done. Otherwise, we repeat this procedure until we end up with $S_p=T_p$
for some $p>1$.  Then $S_p$ becomes the set of representatives we searched for.
\epf

We say that a decomposition of $N_X$ into $N_0$-cosets satisfying the conditions
in Proposition~\ref{pro:cosets} is \emph{good}. 

\medbreak

If $N_X=\bigsqcup_{i=0}^k \sigma_iN_0$ is a good decomposition, then for each $y\in X$ we set 
\begin{equation}
\label{eqn:sigma_y}
\sigma_y\coloneqq\sigma_j.
\end{equation}
for $j\in\{0,\dots,k\}$ such that $\sigma_j\trid x_0=y$.


\begin{lem}\label{lem:c-invariant}
	If $N_X=\bigsqcup_{i=0}^k \sigma_iN_0$ is good, then 
	\begin{align*}
		c(x_0\trid n)=c(n).
	\end{align*}
\end{lem}
\pf
	Indeed, $x_0\trid n=x_0\sigma(n)x_0^{-1}c(n)$, as $c(n)\in N_0$ and
	$x_0\sigma(n)x_0^{-1}=\sigma_i$, for some $i\in\{0,\dots,k\}$.
\epf

Recall the definition of the group homomorphism $\d:G_X\to\Z$ from
\eqref{eqn:d}.  
\begin{lem}
	\label{lem:la_locura}
	For each $u\in G_X$ and $y\in X$, 
	\[
	\sigma_{u\trid y}=
	u\sigma_yx_0^{-\d(u)}
	c\left(u\sigma_yx_0^{-\d(u)}\right)^{-1}.
	\]
	In particular if $n\in N$, then $\sigma_{n\trid y}=
		\sigma\left(n\sigma_y\right)$.
\end{lem}

\begin{proof}
Since 
\[
\sigma_{u\trid y}\trid x_0=u\trid y=u\trid(\sigma_y\trid x_0)=(u\sigma_y)\trid x_0=(u\sigma_yx_0^{-\d(u)})\trid x_0
\]
and $u\sigma_yx_0^{-\d(u)}\in N$, it follows that $\sigma_{u\trid y}=\sigma(u\sigma_yx_0^{-\d(u)})$. Then 
\[
	u\sigma_yx_0^{-\d(u)}
	=
	\sigma_{u\trid y}
	c\left(u\sigma_yx_0^{-\d(u)}\right),
\]
and the first claim follows. If $n\in N$, then $\d(n)=0$ and therefore it follows that $\sigma_{n\trid y}=
	n\sigma_yc\left(n\sigma_y\right)^{-1}=\sigma\left(n\sigma_y\right)$ \emph{cf.} \eqref{eqn:c(n)}.
\end{proof}

We can now proceed to prove Theorem~\ref{thm:reconstruction}.

\begin{proof}[Proof of Theorem~\ref{thm:reconstruction}]
We need to define an inverse to the map
\eqref{eqn:correspondence-proof}. Fix $a\in A$, $g\in\Hom(N_0,A)$ and set
$f:G\to M$ as 
\begin{align*}
	f(u)(y)\coloneqq\d(u)\,a+g\left(c(u\sigma_yx_0^{-\d(u)})\right),
\end{align*}
for each $u\in G$.  We show that $f\in Z^1(G,M)$ and $f\mapsto (a,g)$ via
\eqref{eqn:correspondence-proof}. 

On the one hand, as $\sigma_{x_0}=\sigma_0=1$, 
\begin{align*}
 f(x_0)(x_0)=a+g(c(x_0x_0^{-1}))=a.
\end{align*}
On the other, if $n_0\in N_0$, then $\d(n_0)=0$ and thus
\begin{align*}
 f_0(n_0)=f(n_0)(x_0)=g(c(n_0))=g(n_0).
\end{align*}
Now we check the cocycle condition. First, 
\[
	 f(uu')(y)=\d(uu')a+g\left(c(uu'\sigma_yx_0^{-\d(uu')})\right).
\]
Second, 
\begin{align*}
 (f(u)\cdot u')(y)&+f(u')(y)=f(u)(u'\trid y)+f(u')(y)\\
 &=\d(u)a+g(c(u\sigma_{u'\trid y}x_0^{-\d(u)}))
 +\d(u')a+g(c(u'\sigma_yx_0^{-\d(u')}))\\
 &=f(uu')(y), 
 \end{align*}
since $A$ is abelian, $\d$ and $g$
are a group homomorphisms and 
\[
	c\left(u\sigma_{u'\trid y}x_0^{-\d(u)}\right)=c\left(uu'\sigma_yx_0^{-\d(uu')}\right)c\left(u'\sigma_yx_0^{-\d(u')}\right)^{-1}
\]
by Lemma~\ref{lem:la_locura}. Hence $f\in Z^1(G,M)$.
\end{proof}

\section{Applications}
\label{sec:applications}

Our method for computing the 2nd cohomology group of an indecomposable quandle
$X$ involves the group $N_0$, see~\S\ref{subsection:basics}. In several important cases,
this group can be obtained applying the following lemma.
\begin{lem}
    \label{lem:N0}
    Let $X$ be a finite indecomposable quandle and $x_0\in X$. Assume that the
    canonical map $X\to G_X$ is injective.  Then 
    \[
    N_0\simeq [F_X,F_X]\cap C_{F_X}(\psi(x_0)),
    \]
	where $\psi\colon X\to G_X\to F_X$ is the composition of the canonical maps
	and $C_{F_X}(\psi(x_0))$ is the centralizer of $\psi(x_0)$
	in $F_X$.
\end{lem}

\begin{proof}
    Since $X\to G_X$ is
	injective and $X$ is indecomposable, $X$ can be identified
	with the conjugacy class of $x_0$ in $G_X$. By~\cite[Lemma 1.8]{MR3276225},
	$X$ can also be identified with the conjugacy class of $\psi(x_0)$ in
	$F_X$. From Lemma~\ref{lem:isoclinic} one obtains that
	$N_X=[G_X,G_X]\simeq[F_X,F_X]$ and thus the claim follows.
\end{proof}

\begin{rem}
    If $X$ is a conjugation quandle, then the canonical map $X\to G_X$
    is injective. Thus Lemma~\ref{lem:N0} gives a nice description of 
    $N_0$ in the case of finite indecomposable conjugation quandles.
\end{rem}

\begin{exa}
    The claim of Lemma~\ref{lem:N0} does not hold for arbitrary quandles. 
	Let $X$ be the quandle $\{x_1,x_2,x_3,x_4\}$ with the structure given by 
	\begin{align*}
		\varphi_{x_1}=(x_2x_3x_4), &&
		\varphi_{x_2}=(x_1x_4x_3), &&
		\varphi_{x_3}=(x_1x_2x_4), &&
		\varphi_{x_4}=(x_1x_3x_2).
	\end{align*}
	This quandle is isomorphic to the conjugacy class of $3$-cycles in $\Alt_4$. 
	Let $f\colon X\times X\to\C^\times$ be the map
    given by
    \[
        f(x,y)=\begin{cases}
            1 & \text{if $x=x_1$ or $y=x_1$ or $x=y$,}\\
            -1 & \text{otherwise}.
        \end{cases}
    \]
    Then $f$ is a $2$-cocycle of $X$ with values in $\{-1,1\}\simeq\Z_2$, see
    ~\cite[Example 2.2]{MR1994219}. Let $Y=X\times\{-1,1\}$ be the
    quandle given by 
    \[
        (x,i)\trid (y,j)=(x\trid y,jf(x,y)),\quad
		x,y\in X,\,i,j\in\{-1,1\}.
    \]
	Then the canonical map $Y\to G_Y$ is not injective. Indeed, 
	\[
		(x_2,1)(x_3,-1)=(x_1,1)(x_2,1)=(x_3,1)(x_1,1)=(x_2,1)(x_3,1)
	\]
	implies that 
	$(x_3,-1)=(x_3,1)$ in $G_{Y}$. 

	Fix $y_0\in Y$. A straighforward calculation shows that $F_Y\simeq\SL(2,3)$
	and $[F_Y,F_Y]\cap C_{F_Y}(\psi(y_0))\simeq\Z_2$. However, since
	$[F_Y,F_Y]$ and $Y$ both have eight elements, $N_0$ is the trivial group. 
\end{exa}

\subsection{Transpositions in $\Sym_n$}

Let $X=(12)^{\Sym_n}$ be the quandle of transpositions in the symmetric group
$\Sym_n$.  For $n\geq 4$ a non-constant 2-cocycle $\chi\in H^2(X,\C^\times)$
was constructed in \cite{MR1800714}. This cocycle is given by 
\begin{align}\label{eqn:MS}
	\chi(\sigma,\tau)&=\begin{cases}
		1 & \text{if $\sigma(i)<\sigma(j)$},\\
		-1 & \text{otherwise},
	\end{cases}
\end{align}
where $\tau=(ij)$, $1\leq i<j\leq n$.

\begin{lem}\label{lem:Sn}
	Let $X=(12)^{\Sym_n}$, $n\geq 4$, and fix $x_0=(12)\in X$. 
	\begin{enumerate}
		\item $F_X\simeq \Sym_n$. Hence $N_X\simeq\Alt_n$.
        \item $N_0\simeq \Z_2\ltimes\Alt_{n-2}$. In particular,
            $N_0/[N_0,N_0]\simeq\Z_2$.
	\end{enumerate}
\end{lem}

\begin{proof}
	Recall that $\Sym_n=\langle \sigma_1,\dots,
	\sigma_{n-1}\rangle$ with relations
	\begin{align*}
		&\sigma_i\sigma_{i+1}\sigma_i=\sigma_{i+1}\sigma_i\sigma_{i+1}, && 1\leq 
		i<n-1,\\
		&\sigma_k\sigma_j=\sigma_j\sigma_k, && 1\leq j,k< n, \, |j-k|>1,\\
		&\sigma_i^2=1, &&1\leq i< n.
	\end{align*}
	Set $\iota:X\hookrightarrow \Sym_n$ the canonical inclusion let
	$\varphi:\langle X\rangle\to \Sym_n$ the unique group homomorphism with
	$\varphi_{|X}=\iota$. This is in fact an epimorphism.  Observe that
	$\varphi(x^2)=\iota(x)^2=1$ and
	$$
	\varphi(xy)=\iota(x)\iota(y)=\iota(x)\iota(y)\iota(x)^{-1}\iota(x)=\iota(x\trid y)\iota(x)=\varphi((x\trid y) x).\\
	$$
	Thus, $\varphi$ factors through $\phi:F_X\twoheadrightarrow \Sym_n$. Now,
	set $S$ be the free group on $s_1,\dots,s_{n-1}$ and let $\psi':S\to F_X$
	be the group epihomomorphism given by $s_i\mapsto (i\,i+1)$.  Now $\psi'$
	factors through $\psi:\Sym_n\twoheadrightarrow F_X$ and it is clear that
	$\phi$ and $\psi$ are inverses to each other.

	Let us prove the second claim.  By the first part, we identify $N$ with
	$\Alt_n$.  Consider $\Alt_{n-2}\leq \Alt_n$ as those permutations fixing 1
	and 2 and set $t=(12)(34)$. Then $t\sigma t^{-1}\in\Alt_{n-2}$ for all
	$\sigma\in\Alt_{n-2}$. Clearly $\langle t\rangle\ltimes
	\Alt_{n-2}\leq N_0$. 

	Since $\Alt_n$ is generated by $\{(34\ell)\mid 1\leq\ell\leq n,\;\ell\neq
	3,4\}$, the group $N$ is generated by the subgroups $\Alt_{n-2}$ and
	$\Alt_{4}\simeq\langle (134),(234)\rangle$.  Notice that $\langle (134),(234)\rangle\cap 
	N_0\simeq\langle t\rangle$. 
	We have $|\langle t\rangle\ltimes\Alt_{n-2}|=(n-2)!$ and 
	\[
	\{\sigma (12)\sigma^{-1}:\sigma\in\Alt_n\}=(12)^{\Sym_n}.
	\]
	Thus $|N_0|=|N|/|(12)^{\Sym_n}|=(n-2)!$ and hence $N_0=\langle
	t\rangle\ltimes\Alt_{n-2}$. 

   Finally, since the commutator subgroup of some group $A\ltimes B$ is
    the group generated by $[A,A]\cup [A,B]\cup [B,B]$ and $N_0=\langle
    t\rangle\ltimes\Alt_{n-2}$, it follows that $[N_0,N_0]\simeq\Alt_{n-2}$ 
    and hence $N_0/[N_0,N_0]\simeq\Z_2$.
\end{proof}

\begin{thm}
    \label{thm:FK}
    Let $n\geq4$ and $X=(12)^{\Sym_n}$ be the conjugacy class of
    transpositions. Then
    $H^2(X,\C^\times)\simeq\C^\times\times\langle\chi\rangle$.
\end{thm}

\begin{proof}
    Set $x_0=(12)\in\Sym_n$. 
    Since 
	$N_0\simeq\Z_2\ltimes\Alt_{n-2}$ and $N_0/[N_0,N_0]\simeq\Z_2$ by 
    Lemma~\ref{lem:Sn}, it follows that  
    $\Hom(N_0,\C^\times)\simeq\Z_2$. Applying the
    isomorphism~\eqref{eqn:map} of Theorem~\ref{thm:main} to the $2$-cocycle
    $\chi$ given in~\eqref{eqn:MS},
    \[
        \chi\mapsto (-1,(f_{\chi})_0),
    \]
    where $(f_{\chi})_0\colon N_0\to\C^\times$, $n_0\mapsto f_\chi(n_0)(x_0)$,
	$n_0\in N_0$. Now the claim follows since $(f_\chi)_0$ generates
	$\Hom(N_0,\C^\times)$. Indeed, $f_\chi\ne1$ since 
%
	\begin{align*}
		(f_\chi)_0((12)(34))&=f_\chi((12)(34))(12)\overset{\eqref{eqn:fq}}{=}\chi((12),(34)\trid (12))\chi((34),(12))\\
		&=\chi((12),(12))\chi((34),(12))=-1. 
	\end{align*}
	This completes the proof.
\end{proof}

\subsection{Eisermann formula}\label{sec:eisermann}

We give a new proof of a formula of Eisermann as a consequence of our results. 

\begin{thm}{\cite[Theorem 1.12]{MR3205568}}
	\label{thm:eisermann}
	Let $X$ be a finite indecomposable quandle and $x_0\in X$. Then
	\[
	H_2^Q(X,\Z)\simeq\left([G_X,G_X]\cap
	C_{G_X}(x_0)\right)_\ab\simeq\left(N_0\right)_\ab, 
	\]
	where $N_0$ is the stabilizer
	of a given $x_0\in X$ of the action of $[G_X,G_X]$ on $X$. 
\end{thm}

\begin{proof}
The claim follows by ``chasing'' the chain of equivalences
\[
A\times \Hom(N_0,A)\simeq H^2(X,A)\simeq \Hom(H_2(X,\Z),A)
\]
given by the application of Theorem \ref{thm:reconstruction} and Lemma~\ref{lem:UCT}. More explicitly, 
if $(a,g)\in A\times \Hom(N_0,A)$, then it defines $q\in H^2(X,A)$ via  \eqref{eqn:cociclo2}, which in turn defines a morphism $H_2(X,\Z)\to A$ by Lemma \ref{lem:UCT}:
\[
[x,y]\mapsto q_{x,y}=a+g(c(x\sigma_yx_0^{-1}))\in A,
\]
\emph{cf.} Theorem \ref{thm:reconstruction}. Now, $H_2(X,\Z)\simeq H_2^Q(X,\Z)\times \Z$ by Lemma \ref{lem:H2}
and so this assignment becomes a map in $\Hom(H_2^Q(X,\Z)\times \Z,A)$:
\[
\left([x,y], \ell\right) \longmapsto \ell\,a + g(c(x\sigma_yx_0^{-1})).
\]
Thus we see that the restriction of this map to $H_2^Q(X,\Z)\times \{0\}$ gives an equivalence $\Hom(H_2^Q(X,\Z),A)\simeq \Hom(N_0,A)\simeq (N_0)_{\ab}$ for any abelian group $A$. Hence we derive Eisermann's formula $H_2^Q(X,\Z)\simeq (N_0)_{\ab}$.
\end{proof}

If we combine this fact with Lemma \ref{lem:H2}, we obtain the following. 

\begin{cor}
	\label{cor:H_2}
	Let $X$ be a finite indecomposable quandle, $x_0\in X$. Then
	$H_2(X,\Z)\simeq \left(N_0\right)_\ab\times \Z$. 
\end{cor}

\subsection{Affine quandles}

Let $L$ be an abelian group and $\gamma\in\Aut(L)$. The \emph{affine} (or
\emph{Alexander}) quandle $\Aff(L,\gamma)$ is the set $L$ together with the action 
\[x\trid
y=\gamma(y)+x-\gamma(x), \qquad x,y\in L.
\] 

In \cite{C} Clauwens described the enveloping group of an affine quandle; we
review his construction next.  Set 
\begin{align}\label{eqn:S(L,g)}
\begin{split}
\tau_\gamma&\colon L\otimes_\Z L\to L\otimes_\Z L, \quad (x,y)\mapsto (x,y)-(y,\gamma(x)),\\
	S(L,\gamma) &\coloneqq \coker \tau_\gamma=L\otimes_\Z  L/\langle(x,y)-(y,\gamma(x))\rangle.
\end{split}
\end{align}
We write $[x,y]\in S(L,\gamma)$ for the class of an
element $x\otimes y\in L\otimes_\Z  L$. Set $X=\Aff(L,\gamma)$; then $G_X$ is the set $L\rtimes \Z \times S(L,\gamma)$ with
multiplication
\begin{equation*}
	\left(x,m,[p,q]\right)\left(y,n,[r,s]\right)=\left(x+\gamma^m(y),m+n,[p+r+x,q+s+\gamma^m(y)]\right), 
\end{equation*}
for $m,n\in\Z$, $x,y\in L$, $[p,q], [r,s] \in S(L,\gamma)$. 

The rack $X$ identifies with the subset $L\rtimes \{1\}\times 0$ with the rack action given by conjugation:
\begin{align*}
(x,1,0)(y,1,0)&=(x+\gamma(y),2,[x,\gamma(y)])=(x+\gamma(y),2,[x\trid y,\gamma(x)])\\
&=(x+\gamma(y)-\gamma(x)+\gamma(x),2,[x\trid y,\gamma(x)])\\
&=(x\trid y,1,0)(x,1,0)
\end{align*}
since $[x\trid y,\gamma(x)]=[\gamma(y),\gamma(x)]+[x,\gamma(x)]-[\gamma(x),\gamma(x)]=[x,\gamma(y)]$, as $[x,\gamma(x)]=[\gamma(x),\gamma(x)]$.
We fix $x_0=(0,1,0)$; then
\begin{align}\label{eqn:NX}
	N_X= L\times \{0\}\times S(L,\gamma), && N_0=\{0\}\times \{0\}\times S(L,\gamma).
\end{align}
Let $\{x_0,x_1,\dots,x_n\}$ be an enumeration of the elements of $L$. In particular, 
\begin{align}\label{eqn:decomposition}
	N_X=\bigsqcup\limits_{i\in\{0,\dots,n\}}\sigma_i N_0\simeq L\times \coker\tau_\gamma, \qquad \sigma_i=(x_i,0,0),
\end{align}
is a good decomposition of $N$ into $N_0$-cosets, \emph{cf.} Proposition \ref{pro:cosets}.
Indeed,
\begin{enumerate}
\item $\sigma_0=(0,0,0)$ coincides with the unit element in $G_X$;
\item fix $j\in \{0,\dots,n\}$ and let $k\in \{0,\dots,n\}$ be such that $x_k=\gamma(x_j)$. Then $x_0\trid \sigma_j=(0,1,0)(x_j,0,0)(0,-1,0)=(\gamma(x_j),0,0)=\sigma_k$; and 
\item if $i\in \{0,\dots,n\}$ and $x_j=(1-\gamma)^{-1}(x_i)$, then $\sigma_j\trid x_0=x_i$.
\end{enumerate}
Recall from \eqref{eqn:sigma_y} the definition of the elements $\sigma_y$,
$y\in X$, and from~\eqref{eqn:c(n)} the map $c\colon N_X\to N_0$. We see from Item (3) above that in
this case
\begin{align*}
\sigma_y=\left((1-\gamma)^{-1}(y),0,0\right), \quad y\in X.
\end{align*}
As a direct consequence of Theorem \ref{thm:reconstruction}, we obtain the following.
\begin{pro}
	\label{pro:affine}Let $L$ be an abelian group, $\gamma\in\Aut(L)$ and  
	$X=\Aff(L,\gamma)$ be the corresponding affine quandle and set $\Gamma=S(L,\gamma)$ as in \eqref{eqn:S(L,g)}. 
	Fix $x_0=0\in X$ and let $A$ be an abelian
	group with trivial $G_X$-action. 
	Consider a 
	 decomposition of $N_X$ into $N_0$-cosets as in \eqref{eqn:decomposition}.
  For each $a\in A$ and
	$g\in \Hom(\Gamma,A)$, the map $q\colon X\times X\to A$ given by 
	\begin{align}\label{eqn:explicit-q-affine}
		q_{x,y}=a+\sum\limits_{0<j<\ord(\gamma)}g\left([x, \gamma^j(y)]\right)
	\end{align}
	is a $2$-cocycle of $X$ and any $q\in H^2(X,A)$ arises in this way. 
\end{pro}
\pf
By Theorem \ref{thm:reconstruction} and Corollary \ref{cor:reconstruction}, any 2-cocycle is of the form \begin{align*}
		q_{x,y}=a+g(c(x\sigma_yx_0^{-1})).
	\end{align*}
for some $a\in A$ and $g\in \Hom(\Gamma,A)$. Using the identifications above, we have
\begin{align*}
x\sigma_yx_0^{-1}&=(x,1,0)((1-\gamma)^{-1}(y),0,0)(0,-1,0)\\
&=(x+\gamma(1-\gamma)^{-1}(y),0,[x,\gamma(1-\gamma)^{-1}(y)])\\
&=\sigma_k(0,0,[x,\gamma(1-\gamma)^{-1}(y)])\in\sigma_kN_0
	\end{align*}
for $k\in \{0,\dots,n\}$ such that $x+\gamma(1-\gamma)^{-1}(y)=x_k$. Hence 
\[\gamma(1-\gamma)^{-1}(y)=(1-\gamma)^{-1}(y)-y=\sum\limits_{0<j<\ord(\gamma)}\gamma^j(y)
\]
and the result follows.
\epf

If $L=\F_q$ is the finite field of $q$ elements and $\gamma$ is the
multiplication by some $1\neq\omega\in\F_q^\times$, we write
$\Aff(q,\omega)=\Aff(L,\gamma)$.

\begin{lem}
	\label{lem:afin-p}
	Let $p$ be a prime number and $1\neq \omega \in\F_p^\times$, set
	$X=\Aff(p,\omega)$. Then $G_X\simeq L\rtimes
	\Z$, $N_X\simeq L$ and $N_0$ is trivial.
\end{lem}

\begin{proof}
	Indeed,
	$S(L,\gamma)$ 
	is a quotient of $\Z_p\simeq \Z_p\otimes_\Z  \Z_p$ and we have that $0\neq
	(1-\omega)\otimes 1\in \Ima(\tau_\gamma)$, hence $S(L,\gamma)=0$ and the lemma follows.
\end{proof}

We recover the following result from~\cite[Lemma 5.1]{MR2093034}.  

\begin{pro}
	\label{pro:p}
$
H^2(\Aff(p,\omega),\C^\times)\simeq\C^\times. 
$
\end{pro}
\begin{proof}
	It follows from Theorem \ref{thm:main}, 	using Lemma \ref{lem:afin-p}.
\end{proof}

\subsection{Indecomposable quandles of size $p^2$}

Let $p$ be a prime number and let $X$ be an indecomposable quandle of size $p^2$. By \cite{MR2093034}, $X$ is one of the following affine quandles 
$(L,\gamma)$ in the following list:
\begin{align}
\label{A1} L&=\Z_p\oplus \Z_p, & \gamma_{\alpha,\beta}(x,y)&=(\alpha\,x, \beta\,y), & \alpha,\beta&\in\Z_p^*\setminus\{1\};\\
\label{A2} L&=\Z_p\oplus \Z_p, & \gamma_{\alpha}(x,y)&=(\alpha\,x, \alpha\,y+x), &\alpha&\in\Z_p^*\setminus\{1\};\\
\label{A3} L&=\F_{p^2}, & \gamma_{\alpha}(x)&=\alpha\,x, &\alpha&\in\F_{p^2}\setminus\F_{p};\\
\label{A4} L&=\Z_{p^2},  & \gamma_{\alpha}(x)&=\alpha\,x, &\alpha&\not\equiv0,1\, (p).
\end{align}
We  identify $\F_{p^2}\simeq \F_{p}\oplus \F_p$ as abelian groups for notational reasons. 
For $\alpha=(\alpha_0,\alpha_1)\in\F_p^2$ we set 
\begin{align}\label{eqn:d_alpha}
d_\alpha&\coloneqq (1-\alpha_0+\alpha_1)(1-\alpha_0-\alpha_1)(1-\alpha_0^2+\alpha_1^2).
\end{align}
Assume $\alpha\in \F_{p^2}\setminus\F_{p}$, so $\alpha_1\neq0$. If $d_\alpha=0$, then $\alpha_0\neq 1$ and we set
\begin{align}\label{eqn:t_a}
 t_\alpha&\coloneqq (\alpha_0-\alpha_0^2+\alpha_1^2)(1-\alpha_0)^{-1}, &  s_\alpha&\coloneqq (1-\alpha_0)\alpha_1^{-1}.
\end{align}


\begin{pro}\label{pro:homology-affine}
The 2nd homology groups of the indecomposable quandles of order $p^2$ are as follows:
\begin{align*}
H_2\left((\Z_p\oplus \Z_p, \gamma_{\alpha,\beta}),\Z\right)&\simeq \begin{cases}
                                                     \Z\times\Z_p, &\mbox{if } \alpha\beta=1,\\
\Z, &\mbox{if } \alpha\beta\neq1.
                                                    \end{cases}
\\
H_2\left((\Z_p\oplus \Z_p, \gamma_{\alpha}),\Z\right)&\simeq \begin{cases}
                                                     \Z\times\Z_p, &\mbox{if } \alpha^2=1,\\
\Z, &\mbox{if } \alpha^2\neq 1.
                                                    \end{cases}
\\
H_2\left((\F_{p^2}, \gamma_{\alpha}),\Z\right)&\simeq \begin{cases}
                                                     \Z\times\Z_p, &\mbox{if } d_\alpha = 0,\\
\Z, &\mbox{if } d_\alpha\neq 0.
                                                    \end{cases}
\\
H_2\left((\Z_{p^2}, \gamma_{\alpha}),\Z\right)&\simeq \Z.
\end{align*}
\end{pro}

\pf
By Corollary~\eqref{cor:H_2}, if $X=(L,\gamma)$ and  $\tau_\gamma:L\otimes_\Z  L\to L\otimes_\Z  L$ as in \eqref{eqn:S(L,g)}, then 
$$
H_2(X,\Z)=(N_0)_{\ab}\times \Z=\coker\tau_\gamma\times \Z
$$
We compute $\coker\tau_\gamma$ case by case. We will use the identifications
\begin{equation}
\label{eqn:identifications}
\begin{aligned}
\big(\Z_p\oplus \Z_p\big)\otimes_\Z  \big(\Z_p\oplus \Z_p\big)&\simeq \Z_p^4, && (a,b)\ot (c,d)\mapsto 
(ac,ad,bc,bd)\\
\F_{p^2}\otimes_\Z  \F_{p^2}\simeq \F_{p}^2\ot_{\F_p} \F_{p}^2&\simeq \F_p^4, && (a,b)\ot (c,d)\mapsto 
(ac,ad,bc,bd)\\
\Z_{p^2}\otimes_\Z  \Z_{p^2}&\simeq \Z_{p^2}, && a\ot b\mapsto ab.
\end{aligned}
\end{equation}

Case \eqref{A1}: We have that $\tau_{\alpha,\beta}\coloneqq\tau_{\gamma_{\alpha,\beta}}$ is
$$
\tau_{\alpha,\beta}((a,b)\ot (c,d))=(a,b)\ot (c,d)-(c,d)\ot (\alpha\,a,\beta\,b).
$$
With the identifications above this yields 
$$
\tau_{\alpha,\beta}:\Z_p^4\to \Z_p^4, \quad (x,y,z,w)\mapsto ((1-\alpha)x,y-\beta\,z,z-\alpha\,y,(1-\beta)\,w).
$$
Next, we compute the image $I_{\alpha,\beta}$ of this map: For 
 $(a,b,c,d)\in\Z_p^4$ to be in this subgroup, we need $x=a(1-\alpha)^{-1}$, $w=d(1-\beta)^{-1}$ (recall $\alpha,\beta\neq 1$) and $y,z$ to be a 
solution of $y-\beta\,z=b$, $-\alpha\,y+z=c$.
This system has always a solution if $\alpha\beta\neq 1$. If $\alpha\beta=1$, then 
\begin{align*}
I_{\alpha,\beta}&=\{(a,b,-\alpha\,b,d)|a,b,d\in\Z_p\}\simeq \Z_p^3, \text{ hence }\\
 \coker\tau_\alpha&=\begin{cases}
0, &\mbox{if } \alpha\beta\neq 1,\\
\Z_p, &\mbox{if } \alpha\beta= 1.
\end{cases}
\end{align*}
In case \eqref{A2}, we have $\tau_\alpha\coloneqq\tau_{\gamma_\alpha}:\Z_p^4\to \Z_p^4$ is given by
\begin{align*}
 (x,y,z,w)\mapsto \big((1-\alpha)x,y-\alpha\,z+x,z-\alpha\,y,(1-\alpha)w-y\big).
\end{align*}
For $(a,b,c,d)$ to be in the image $I_\alpha$ of $\tau_\alpha$, we need $x=a(1-\alpha)^{-1}$ (recall $\alpha\neq1$) and $(y,z,w)$ to be a solution of
\begin{align*}
 &y-\alpha\,z=b-a(1-\alpha)^{-1}, && -\alpha\,y+z=c, && -y+(1-\alpha)w=d.
\end{align*}
This system has always a solution if $\alpha^2\neq 1$. If $\alpha^2=1$, then
\begin{align*}
I_\alpha&=\{(a,b,\alpha\,b-\alpha(1-\alpha)a,d)|a,b,d\in\Z_p\}\simeq \Z_p^3, \text{ hence }\\
\coker\tau_\alpha&=\begin{cases}
                   0, &\mbox{if } \alpha^2\neq 1,\\
\Z_p, &\mbox{if } \alpha^2= 1.
                  \end{cases}
\end{align*}
In case \eqref{A3}, if $\alpha=(\alpha_0,\alpha_1)\in\F_{p}^2\setminus \F_p$ (hence $\alpha_1\neq 0$), then the map 
$\tau_\alpha\in \End(\F_p^4)$ is represented by the matrix 
$$
[\tau_\alpha]=\left(\begin{smallmatrix}
                     1-\alpha_0&0&-\alpha_1&0\\
                     -\alpha_1&1&-\alpha_0&0\\
                     0&-\alpha_0&1&-\alpha_1\\
                     0&-\alpha_1&0&1-\alpha_0\\
                    \end{smallmatrix}
\right),
$$
with $\det [\tau_\alpha]=d_\alpha$, see \eqref{eqn:d_alpha}. Let $I_\alpha$ denote the image of this map. Now, the rank of this matrix is $\geq 3$, as  $\det\left(\begin{smallmatrix}
                     0&-\alpha_1&0\\
                     1&-\alpha_0&0\\
                     -\alpha_0&1&-\alpha_1
                    \end{smallmatrix}
\right)=-\alpha_1^2\neq0$. Hence,
$$
\coker\tau_\alpha=\begin{cases}
                   0, &\mbox{if } \det [\tau_\alpha]\neq 0,\\
\Z_p, &\mbox{if } \det [\tau_\alpha]= 0.
                  \end{cases}
$$
If $ \det [\tau_\alpha]=0$, i.e. $\d_\alpha=0$, then we set $t_\alpha, s_\alpha\in\F_p$ as in \eqref{eqn:t_a} and thus
\[I_\alpha=\{(a,b,c,-t_\alpha(a+b)-s_\alpha\,c)|a,b,c\in\Z_p\}\simeq \Z_p^3.
\] 
In case \eqref{A4}, $\tau_\alpha:\Z_{p^2}\to \Z_{p^2}$ is  $x\mapsto (1-\alpha)x$; hence
$
\coker\tau_\alpha=0$.
\epf

\subsection{Explicit cocycles}

Next we apply Proposition \ref{pro:affine} to compute all non-constant
2-cocycles for the affine quandles $X$ described in \eqref{A1}--\eqref{A4}.
More precisely, we focus on those affine quandles in that list admitting a
non-constant 2-cocyle, as stated in Proposition \ref{pro:homology-affine}:
\begin{align}
\label{A1'} L&=\Z_p\oplus \Z_p, & \gamma_{\alpha}(x,y)&=(\alpha\,x, \alpha^{-1}\,y), &\alpha&\in\Z_p^*\setminus\{1\};\\
\label{A2'} L&=\Z_p\oplus \Z_p, & \gamma(x,y)&=(-x, x-y);&&\\
\label{A3'} L&=\F_{p^2}, & \gamma_{\alpha}(x)&=\alpha\,x, &\alpha&\in\F_{p^2}\setminus\F_{p}, d_\alpha=0.
\end{align}
Recall our identification $\F_{p^2}\simeq \F_{p}^2$, $x\mapsto (x_0,x_1)$, and $t_\alpha,s_\alpha\in\Z_p$ from $\eqref{eqn:t_a}$.
For $x,y\in L$ and $j\in\N$ we set, for $X$ is as in~\eqref{A1'}, 
\[
\zeta_j(x,y)=
\alpha^jx_2y_1+\alpha^{1-j}x_1y_2; 
\]
for $X$ is as in~\eqref{A2'}, 
\[
\zeta_j(x,y)=(j+2(-1)^j)x_1y_1+(-1)^j(x_1y_2-x_2y_1); 
\]
and, for $X$ is as in~\eqref{A3'}, 
\[
\zeta_j(x,y)=x_1(\alpha^j y)_1+t_\alpha\left(x_0(\alpha^j y)_0+x_0(\alpha^j y)_1\right)
+s_\alpha\,x_1(\alpha^j y)_0. 
\]
Next, we define the map $\lg \, , \,\rg:L\times L\to \Z$ as
\[
\lg  x, y \rg =\sum_{0< j <\ord(\gamma)}\zeta_j(x,y), \quad x,y\in L.
\]
Notice that $\ord(\gamma)=p-1$, $2p$ (or 2 if $p=2$) or $p^2-1$ according to whether $X$ is as in \eqref{A1'}, \eqref{A2'} or \eqref{A3'}, respectively.
\begin{pro}\label{pro:explicit-affine}
Let $X=(L,\gamma)$ be an indecomposable affine rack of order $p^2$. If $q\in H^2(X,\k^*)$ is non-constant, then $X$ belongs to the list \eqref{A1'}--\eqref{A3'} and 
there are $0<\ell<p$ and $\lambda\in\k^*$ such that
\begin{align}\label{eqn:expl-cocycle}
q_{x,y}=\lambda\exp\left(\frac{2\pi\mbox{i}\ell}{p}\lg x,y\rg \right), \quad x,y\in X.
\end{align}
\end{pro}
\pf
Fix $x_0=0\in L$ and a good decomposition $N\simeq L\times \coker\tau_\gamma$ of $N_X$ into $N_0$-cosets, see \eqref{eqn:decomposition}. In this case, $N_0=x_0\times \coker\tau_\gamma\simeq \Z_p$, by 
Proposition \ref{pro:homology-affine}. More precisely, if we denote by $\varphi:N_0\to \Z_p$ this isomorphism, 
then  it follows from the proof of Proposition \ref{pro:homology-affine} that, for $t_\alpha$, $s_\alpha$ as in \eqref{eqn:t_a}: 
\begin{align}\label{eqn:classes}
\varphi\left([(a,b), (c,d)]\right)=\begin{cases}
                                        bc+\alpha ad \in\Z_p, & X\text{ as } \eqref{A1'};\\
                                        bc+ ad-2ac \in\Z_p,   & X\text{ as } \eqref{A2'};\\
                                        bd+t_\alpha(ac+ad)+s_\alpha\,bc  \in\Z_p, & X\text{ as } \eqref{A3'}.
                                        \end{cases}
\end{align}
On the other hand, if $g\in\Hom(N_0,\k^*)$, then there is $0\leq \ell<p$ such that $g$ is the morphism $g_\ell$ given by $1\mapsto \exp\left(\frac{2\pi\mbox{i}\ell}{p}\right)$. By Proposition \ref{pro:affine}, any $q\in H^2(X,\k^*)$ is thus of the form
\begin{align*}
q_{x,y}=\lambda\prod\limits_{0<j<\ord(\gamma)}\exp\left(\frac{2\pi\mbox{i}\ell}{p}\varphi([x,\gamma^j(y)]) \right), \quad x,y\in X,
\end{align*}
for some $\lambda\in\k^*$, $\ell\in\Z$.
Hence the result follows as $\zeta_j(x,y)\in\Z$ is a representative of $\varphi([x,\gamma^j(y)])$, for each $x,y\in X$, via \eqref{eqn:classes}.
\epf

\section*{Acknowledgements}

We thank G. Garc\'ia and M. Kotchetov for interesting discussions.  We also
thank N. Andruskiewitsch for his constant guidance and support.  This work was
initiated while the authors were visiting Mar\'ia Ofelia Ronco, at Universidad
de Talca, Chile. We are grateful for her warm hospitality. The
authors are grateful to the reviewer for useful remarks, interesting
suggestions and corrections.

\def\cprime{$'$}

\end{document}